\documentclass[12pt]{article}
\usepackage{amsmath,amssymb,amsthm}
\usepackage{hyperref}
\usepackage{amsrefs}

\newcommand{\ceiling}[1]{\ensuremath{\lceil #1 \rceil}}

\newcommand{\RR}{\ensuremath{{\mathbb R}}}
\newcommand{\ZZ}{\ensuremath{{\mathbb Z}}}

\renewcommand{\vec}{\overline}
\newcommand{\vt}{\vec\omega}
\newcommand{\va}{\vec\alpha}

\newcommand{\Expect}[2]{\ensuremath{{\mathbb E}_{#1}\left[#2\right]}}
\newcommand{\Prob}[2]{\ensuremath{{\mathbb P}_{#1}\left[#2\right]}}
\newcommand{\Floor}[1]{\left\lfloor #1 \right\rfloor}

\DeclareMathOperator{\vol}{{\bf vol}}

\DeclareMathOperator{\BOX}{{\textsc{Box}}}
\DeclareMathOperator{\annuli}{{\textsc{Annuli}}}
\DeclareMathOperator{\ball}{{\textsc{Ball}}}

\newtheorem{theorem}{Theorem}
\newtheorem{proposition}{Proposition}
\newtheorem{lemma}{Lemma}

\newtheorem{corollary}{Corollary}

\title{Sets of integers that do not contain \\ long arithmetic progressions}
\author{Kevin O'Bryant}
\date{\today}

\begin{document}
\maketitle

\begin{abstract}
In 1946, Behrend gave a construction of dense finite sets of integers that do not contain a 3-term arithmetic progression (AP). In 1961, Rankin generalized Behrend's construction to sets avoiding $k$-term APs. In 2008, Elkin refined Behrend's 3-term construction, and later in 2008, Green \& Wolf found a distinct approach (albeit morally similar) that is technically more straightforward. This work combines Elkin's refinement and Rankin's generalization in the Green \& Wolf framework. A curious aspect of the construction is that we induct through sets that do not contain a long polynomial progression in order to construct a set without a long AP. The bounds for $r_k(N)$, the largest size of a subset of $\{1,2,\dots,N\}$ that does not contain a $k$ element AP, are (where $\log=\log_2$, for sufficiently large $N$, with $n=\ceiling{\log k}$)
    \[r_3(N) \geq  N \left(\tfrac{\sqrt{360}}{e \pi^{3/2}}-\epsilon\right) \frac{\sqrt[4]{2\log N}}{4^{\sqrt{2 \log N}}},\]
    \[r_k(N) \gg  N \, 2^{-n 2^{(n-1)/2} \sqrt[n]{\log N}+\frac{1}{2n}\log\log N}.\]
The improvement over earlier work is in the simplification of the construction, the explicitness of the bound for $r_3$, and in the $\log\log$ term for general $k$.
\end{abstract}

We denote by $r_k(N)$ the maximum possible size of a subset of $\{1,2,\dots,N\}$ that does not contain $k$ numbers in arithmetic progression. Behrend~\cite{Behrend} proved that
    \[\frac{r_3(N)}{N} \geq C \frac{1}{2^{2\sqrt 2(1+\epsilon) \sqrt{\log N}}},\]
where $\log$ is the base-2 logarithm and each occurrence of $C$ is a new positive constant. Sixty years later, Elkin~\cite{Elkin} strengthened this to show that there are arbitrarily large $N$ satisfying
    \[ \frac{r_3(N)}{N}\geq  C \frac{(\log N)^{1/4}}{2^{2\sqrt 2 \sqrt{\log N}}}, \]
and shortly afterwards Green \& Wolf~\cite{GreenWolf} arrived at the same bound by a different method. For $k\geq 1+2^{n-1}$, Rankin~\cite{Rankin} proved that for each $\epsilon>0$, if $N$ is sufficiently large then
    \[ \frac{r_k(N)}{N} \geq C \frac{1}{2^{{n\,2^{(n-1)/2} \,(1+\epsilon)\,\sqrt[n]{\log N}}}},\]
where $n=\ceiling{\log k}$. This was subsequently rediscovered in a simpler, but less precise, form by {\L}aba \& Lacey~\cite{LabaLacey}. Together with the obvious $r_k(N)\leq r_{k+1}(N)$, these are asymptotically the thickest known constructions. The primary interest in the current work is the following corollary of our main theorem.
\begin{corollary}\label{cor}
Set $n=\ceiling{\log k}$. There exists a positive constant $C$ such that for all $N\geq 1$
    \[\frac{r_k(N)}{N} \geq C \frac{\sqrt[2n]{\log N}}{2^{n\,2^{(n-1)/2} \sqrt[n]{\log N}}}.\]
For every $\epsilon>0$, if $N$ is sufficiently large then
    \[\frac{r_3(N)}{N} \geq \left(\frac{\sqrt{360}}{e \pi^{3/2}}-\epsilon\right) \frac{\sqrt[4]{2\log N}}{2^{2\sqrt{2 \log N}}}.\]
\end{corollary}

Szemer\'{e}di's Theorem states that $r_k(N)=o(N)$, and the task of getting quantitative upper bounds on $r_k(N)$ has been mathematically fruitful. The currently-best upper bounds on $r_k(N)$ are due to Bourgain~\cite{Bourgain}, Green \& Tao~\cite{r4paper}, and Gowers~\cite{Gowers}, respectively:
    \begin{align*}
    r_3(N) &\ll N (\log \log N)^2(\log N)^{-2/3};\\
    r_4(N) &\ll N e^{-C\sqrt{\log\log N}}; \\
    r_k(N) &\ll N (\log \log N)^{-2^{-2^{k+9}}}.
    \end{align*}

It is natural to speculate as to whether the upper or lower bound on $r_k(N)$ is closer to the truth. Certainly, the upper bounds have seen a steady stream of substantive improvements, while the main term of the lower bound has remained unchanged for 50 years. The reader is directed to a discussion on Gil Kalai's blog~\cite{KalaiBlog} for some relevant speculative remarks of Gowers and of Kalai's. 

To prove our result we need to induct through sets that do not contain more elaborate types of progressions. A $k$-term $D$-progression is a sequence of the form
    \[Q(1),Q(2),\dots,Q(k)\]
where $Q$ is a nonconstant polynomial with degree at most $D$. For example, $1$-progressions are proper arithmetic progressions. The sequences $2,1,2,5,10$ and $1,2,4,7,11$ are 5-term $2$-progressions arising from the polynomials $(j-2)^2+1$ and $(j^2-j+2)/2$. In particular, a progression of integers may contain the same number in different places, and may arise from a polynomial whose coefficients are not integers. Also, note that the class of $k$-term $D$-progressions is invariant under both translation and dilation. Let \(r_{k,D}(N)\) denote the maximum possible size of a subset of $[1,N]\cap\ZZ$ that does not contain any $k$-term $D$-progressions.

\begin{theorem}\label{thm:main}
 Fix positive integers $k,D$ and set $n=\ceiling{\log(k/D)}$. There exists a positive constant $C$ such that for every $N$
    \[\frac{r_{k,D}(N)}{N} \geq C \cdot \frac{\sqrt[2n]{\log N}}{2^{n 2^{(n-1)/2} D^{(n-1)/n}\sqrt[n]{\log N}}}.\]
\end{theorem}

To explain what is new and interesting in the current work, we begin by summarizing the earlier constructions. Behrend's construction~\cite{Behrend}, while no longer the numerically best or most general, remains the most elegant. His initial observation is that a sphere cannot contain a 3-term arithmetic progression simply because a line and a sphere cannot intersect more than twice. Let $S$ be a set of points in $\ZZ^d$ all lying on one sphere and having all coordinates positive and smaller than $P$, and then let $A$ be the image of $S$ under the map $\varphi:\langle x_1,\dots,x_d\rangle \mapsto \sum_{i=1}^d x_i (2P)^{i-1}$. Because  $0<x_i<P$, addition of two elements of $A$ will not involve any carrying. This $\varphi$ is therefore a Freiman 2-isomorphism between $S$ and $A$; that is, $\vec x_1+\vec x_2=\vec x_3+\vec x_4$ if and only if $\varphi(\vec x_1)+\varphi(\vec x_2)=\varphi(\vec x_3)+\varphi(\vec x_4)$. Since three integers $a<b<c$ are in arithmetic progression if and only if $a+c=b+b$, this proves that $A$ is free of 3-term arithmetic progressions. The only remaining work is to show that there exists a suitably large $S$, which Behrend did with the pigeonhole principle, and to optimize $P$ and $d$ in terms of $N$.

Rankin combined three observations. His first observation was that Behrend's use of the pigeonhole principle could be replaced with a number-theoretic result on the number of representations of a huge number as a sum of a large number of squares. The second is that a degree $D$ polynomial cannot intersect a sphere in more than $2D$ points, and so Behrend's argument actually gives a lower bound on $r_{2D+1,D}$. The third is that one can use a set that does not contain $k$-term $2D$-progressions to build $S$ as a union of concentric spheres with skillfully chosen radii. The corresponding set $A$ (after mapping $S$ as per Behrend, but with the radix $2P$ replaced by something much larger) will necessarily be free of $k$-term $D$-progressions. This provided for an inductive bound. For example, $r_9=r_{9,1}$ is bounded in terms of $r_{9,2}$, which is bounded in terms of $r_{9,4}$, which is then bounded using Rankin's generalized Behrend argument.

Elkin~\cite{Elkin} improved Behrend's 3-term construction in two ways. First, he used the central limit theorem (and the pigeonhole principle) to guarantee the existence of a large $S$; and second, he considered lattice points in a very thin annulus. Using an annulus instead of a sphere leads to a set $S$ that is substantively larger but, unfortunately, does have 3-term arithmetic progressions. After removing a small number of points to eliminate the progressions, Elkin proceeded along the same line as Behrend, needing to optimize $d$, $P$, and also the thickness of the annulus.

Green \& Wolf~\cite{GreenWolf} recast Elkin's argument in a way that avoids counting lattice points. In the $d$-dimensional torus, they take $S$ to be the intersection of a small box and an annulus. Using random elements $\vt,\va$ of the torus, they consider the map $\varphi: n \mapsto n\,\vt+\va$. Letting $A:=\{a\colon \varphi(a)\in S\}$, this map is a Freiman 2-isomorphism between $A$ and $\varphi(A)$. The randomness allowed them to easily count the size of $A$ and the number of progressions in $A$ that need to be removed.

In the current work we recast Rankin's argument using the lessons of Elkin and Green \& Wolf. We avoid Rankin's sum-of-squares number theory lemma by taking random $\vt,\va$ (we still need the pigeonhole principle, however). We find the right generalization of ``an arithmetic progression in a thin annulus has a small difference'' to $D$-progressions, and thereby generalize Elkin's result to improve Rankin's bound on $r_{2D+1,D}$. Finally, by taking concentric annuli, we smooth out Rankin's inductive step. We note also that previous work has sometimes suffered\footnote{Some would say benefitted.} from a cavalier treatment of error terms. For example, Elkin's ``arbitrarily large $N$'' and Rankin's ``$1+\epsilon$'' term can be eliminated with a little care. We have taken the opposite tack here, in places working for coefficients that are not important in the final analysis, but which we consider to be of interest. In particular, the refinement for $r_3$ stated in Corollary~\ref{cor} constitutes about 15\% (by volume) of this work.

\section{Notation}

Throughout, $\log$ and $\exp$ refer to the base-2 logarithm and exponential. Vectors are all given overlines, as in $\vec x$, and all have dimension $d$.

The parameters $N$ and $d$ tend to infinity together, with $N$ much larger than $d$, and all little-oh notation is with respect to $N$ and $d$. The parameter $d$ is a dimension, and must be an integer, while $N$ need not be an integer. The other fundamental parameters, the integers $k$ and $D$, are held constant.
    
We define the difference operator $\Delta$ to be the map taking a finite sequence $(a_i)_{i=1}^k$ to the finite sequence $(a_{v+1}-a_v)_{v=1}^{k-1}$. The formula for repeated differencing is then
            \[\Delta^n (a_i) =\left( \sum_{i=0}^{n} \binom{n}{i}(-1)^{i} a_{i+v}\right)_{v=1}^{k-n}.\]
We note that a nonconstant sequence $(a_i)$ with at least $D+1$ terms is a $D$-progression if and only if $\Delta^{D+1}(a_i)$ is a sequence of zeros. If $a_i=p(i)$, with $p$ a polynomial with degree $D$ and lead term $p_D$, then $\Delta^D (a_i) = (D! p_D)$, a constant sequence. Note also that $\Delta$ is a linear operator. Finally, we will make repeated use of the fact, provable by induction for $1\leq n \leq k$, that
            \[|\Delta^n (a_i)| \leq  2^{n-1} \left( \max_i a_i - \min_i a_i\right).\]
    
A $k$-term type-$(n,a,b)$ progression is a nonconstant sequence $a_1,a_2,\dots,a_k$ with $k\geq n$, $a_1=a$, and $n$-th differences $\Delta^n(a_i)$ the constant nonzero sequence $(b)$. For example, if $p$ is a degree $n$ polynomial (with lead term $p_n\not=0$) and $k\geq n$, then $p(1),\dots,p(k)$ is a type $(n,p(1),n!p_n)$ progression.
    
The open interval $(a-b,a+b)$ of real numbers is denoted $a\pm b$. The interval $[1,N]\cap\ZZ$ of natural numbers is denoted $[N]$. For positive integers $i$, the box $(\pm 2^{-i-1})^d$, which has Lebesgue measure $2^{-id}$, is denoted $\BOX_i$. We define $\BOX_0=[-1/2,1/2)^d$, and define $\vec x \bmod \vec 1$ to be the unique element $\vec y$ of $\BOX_0$ with $\vec x - \vec y \in \ZZ^d$.
    
A point $\vec x=\langle X_1,\dots, X_d \rangle$ chosen uniformly from $\BOX_D$ has components $X_i$ independent and uniformly distributed in $(-2^{-D-1},2^{-D-1})$. Therefore, $\|\vec x\|_2^2 = \sum_{i=1}^d X_i^2$ is the sum of $d$ iidrvs, and is therefore normally distributed as $d\to\infty$. Further $\|\vec x \|_2^2$ has mean $\mu_D:=2^{-2D} d/12$ and variance $\sigma_D^2:=2^{-4D}d/180$.
    
For any set $A \subseteq[n]$, positive integer $D$, and sufficiently small positive real number $\delta$, we define $\annuli(A ,n ,D,\delta)$ in the following manner:
            \[\annuli(A ,n ,D,\delta) 
                :=\left\{\vec x \in \BOX_D \colon  
                    \frac{\|\vec x\|_2^2-\mu_D}{\sigma_D} \in \bigcup_{a\in A} \left(z-\frac{a -1}{n } \pm \delta\right) \right\},\]
where $z\in \mu_D\pm\sigma_D$ is chosen to maximize the volume of $\annuli(A,n,D,\delta)$. Geometrically, $\annuli(A ,n ,D,\delta)$ is the union of $|A|$ spherical shells, intersected with $\BOX_D$.

\section{Lemmas}

The following lemma is best-possible for $k=2D+1$. Improving the bound for larger $k$ comes down to the following problem: if $Q$ has degree $D$ and all of $|Q(1)|,\dots,|Q(k)|$ are less than 1, then how big can the leading coefficient of $Q$ be?

\begin{lemma}[Sphere-ish polynomials have small-ish lead coefficients] \label{lem:RL1} Let $\delta,r$ be real numbers with $0\leq \delta \leq r$, and let $k,D$ be integers with $D\geq 1, k\geq 2D+1$. If $\vec P(j)$ is a polynomial with degree $D$, and $r-\delta \leq \| \vec P(j) \|_2^2 \leq r+\delta$ for $j\in[k]$, then the lead coefficient of $\vec P$ has norm at most $2^D \left.{(2D)!}\right.^{-1/2} \, \sqrt{\delta}$.
\end{lemma}

\begin{proof}
In this paragraph we summarize the proof; in subsequent paragraphs we provide the details. $Q(j):=\|\vec P(j)\|_2^2 -r$ is a degree $2D$ polynomial of $j$, and each of the $2D+1$ real numbers $Q(1),\dots,Q(2D+1)$ are close to zero. If they were all exactly zero, then $Q$ would have more zeros than its degree and so would necessarily be identically zero. Just having that many values close to 0, however, is already enough to guarantee that the lead coefficient of $Q$ is small.

Let $\vec P(j) =\vec P_{0} +\vec P_{1} j + \dots+\vec P_{D} j^D$. We work with the degree $2D$ polynomial \[Q(j):= \|\vec P(j)\|_2^2 -r=\sum_{n=0}^{2D} q_n j^n,\] and note in particular that $q_{2D} =\|\vec P_{D}\|_2^2$. As $0\leq \delta\leq r$, we conclude that $|Q(j)| \leq \delta$.

Set $\vec q,\vec Q$ to be the column vectors $\langle q_0,q_1,\dots,q_{2D}\rangle^T,\langle Q(1),\dots,Q(2D+1)\rangle^T$, respectively. Let $M$ be the $(2D+1)\times(2D+1)$ matrix whose $(i,j)$-component is $i^{j-1}$. We have the system of equations
    \[ M \, \vec q = \vec Q,\]
which is nonsingular because $M$ is a Vandermonde matrix. By Cramer's rule, the cofactor expansion of a determinant along the last column, and the triangle inequality,
    \begin{equation*}
    q_{2D} =\frac{\det(M')}{\det(M)}
        = \frac{1}{\det(M)} \sum_{j=1}^{2D+1} Q(j) (-1)^{j+1} M_{j,2D+1}
        \leq \frac{\sum_{j=1}^{2D+1} |M_{j,2D+1}|}{|\det(M)|} \delta.
    \end{equation*}
By the formula for the determinant of a Vandermonde matrix (the relevant minors of $M$ are also Vandermonde matrices), we find that
    \begin{equation*}
    \|\vec P_{D}\|_2^2=q_{2D} \leq \frac{\sum_{j=1}^{2D+1} |M_{j,2D+1}|}{|\det(M)|} \,\delta= \frac{2^{2D}}{(2D)!}\,\delta,
    \end{equation*}
completing the proof.
\end{proof}

\begin{lemma}[Tight modular progressions are also non-modular progressions] \label{lem:RL2} Suppose that $p(j)$ is a polynomial with degree $D$, with $D$-th coefficient $p_D$, and set $\vec x_{j} := \vt\, p(j)+\va \mod \vec1$. If $\vec x_{1},\vec x_2,\dots, \vec x_k $ are in $\BOX_D$ and $k  \geq D+2$, then there is a vector polynomial $\vec P(j) = \sum_{i=0}^D \vec P_i j^i$ with $\vec P(j)=\vec x_j$ for $j\in[k]$, and $D!\vec P_D =  \vt \,D! p_D \bmod \vec 1$.
\end{lemma}

\begin{proof}
Since $p$ has degree $D$, the $(D+1)$-th differences of $p(1),p(2),\dots,p(k )$ are zero, and therefore the $(D+1)$-th differences of $\vec x_1,\vec x_2,\dots, \vec x_k $ are $\vec 0$ modulo $\vec1$, i.e., all of their components are integers. We will show that in fact all of their components are strictly between $-1$ and $1$, and so they must all be 0.

The $(D+1)$-th differences are given by (valid only for $1\leq v \leq k -D-1$)
    \[\Delta^{D+1}(\vec x_i)(v) =\sum_{i=0}^{D+1} \binom{D+1}{i} (-1)^i \vec x_{v+i}.\]
Denote the $i$-th component of $\vec x_{j}$ by $x_j^{(i)}$. As $\vec x_{v+i}\in\BOX_D$, each component of $\vec x_{v+i}$ is in $\left(-2^{-D-1},2^{-D-1}\right)$. Thus, the $h$-th component of $\Delta^{D+1}(\vec x_i)(v)$ satisfies
    \[ \left| \sum_{i=0}^{D+1} \binom{D+1}{i} (-1)^i \vec x_{v+i}^{(h)} \right|
                \leq  \sum_{i=0}^{D+1} \binom{D+1}{i} |\vec x_{v+i}^{(h)}|
                < \sum_{i=0}^{D+1} \binom{D+1}{i} 2^{-(D+1)} = 1,\]
and therefore $\Delta^{D+1}(\vec x_i)=(0)$.

Now,
    \[D!\vec P_D = \Delta^D(\vec P(i)) = \Delta^D(\vec x_{i}) \equiv \vt D! p_D \pmod{\vec 1}.\]
As $\vec P(i) \in \BOX_D$ for $1\leq i \leq k $, the above binomial-coefficient triangle-inequality argument tells us that the components of $\Delta^D(\vec P(i))$ are between $-1/2$ and $1/2$, and so $D!\vec P_D =\left( \vt D! p_D \bmod \vec 1\right)$.
\end{proof}

\begin{lemma}[$\annuli$ has large volume] \label{lem:annuli}
If $d$ is sufficiently large, $A\subseteq[n]$, and $2\delta\leq 1/n$, then the volume of $\annuli(A,n,D,\delta)$ is at least
    \(\displaystyle \frac25 \, 2^{-dD} |A| \delta .\) Provided that $\delta \log d \to 0$, the volume of $\annuli(\{1\},1,D,\delta)$ is at least \mbox{$ (\sqrt{2/\pi}-o(1)) \,2^{-dD}\,\delta $}.
\end{lemma}

\begin{proof}
A uniformly chosen element $\vec x=\langle X_1,\dots,X_d\rangle$ of $\BOX_D$ has the $X_i$ independent and each uniformly distributed in $(-2^{-D-1},2^{-D-1})$. Thus $\|\vec x \|_2^2$ is the sum of $d$ iidrvs and has mean $\mu_D:=2^{-2D} d/12$ and variance $\sigma_D^2:= 2^{-4D} d/180$. By the central limit theorem (CLT), the random variable
    \[\frac{\|\vec x \|_2^2 - \mu_D}{\sigma_D}\]
has a normal distribution, as $d\to\infty$, with mean 0 and variance 1. We would like to argue that
    \begin{multline*}
    \vol \annuli(\{1\},1,D,\delta)
        \geq 2^{-dD}\left(\int_{-\delta}^{\delta} \frac{e^{-x^2/2}}{\sqrt{2\pi}}dx \right)
        \geq 2^{-dD}  \left( 2\delta \frac{e^{-\delta^2/2}}{\sqrt{2\pi}}\right)\\
        = 2^{-dD}\delta \left(\sqrt{\frac 2\pi}-o(1)\right),
    \end{multline*}
but we cannot apply the CLT to an interval that is shrinking as rapidly as $\pm \delta$. We get around this by applying the CLT to an interval that shrinks very slowly, and then using an analytic form of the pigeonhole principle to guarantee an appropriately short subinterval with the needed density.

We could accomplish this using only the classical CLT, but it is expeditious to use the quantitative CLT known as the Berry-Esseen theorem~\cite{Feller}*{Section XVI.5}, which is applicable since
    \[\rho_D:=\Expect{}{|X_i^2-2^{-2D}/12|^3}=2^{-6D}(3+2\sqrt3)/11340<\infty.\]
Let $I$ be an interval whose endpoints depend on $d$. The Berry-Esseen theorem implies that
    \[\Prob{}{\frac{\|\vec x \|_2^2 - \mu}{\sigma_D}\in I} \geq \frac{1}{\sqrt{2\pi}}\int_I \exp(-x^2/2)\,dx - 2\,\frac{\rho_D}{(\sigma_D/\sqrt{d})^3\sqrt{d}}.\]

First we handle the case $A=\{1\},n=1$. We have
    \begin{align*}
        \Prob{}{\frac{\|\vec x \|_2^2 - \mu_D}{\sigma_D}\in \pm \frac{1}{\log d}}
            &\geq  \frac{1}{\sqrt{2\pi}}\int_{-1/\log d}^{1/\log d} \exp(-x^2/2)\,dx - 2\,\frac{\rho_D}{(\sigma_D/\sqrt{d})^3\sqrt{d}} \\
            &\geq \frac{1}{\sqrt{2\pi}}\, \frac{2}{\log d} \exp\!\big(-(1/\log d)^2/2\big) - \frac{3}{\sqrt{d}}\\
            &\geq \frac{\sqrt{2/\pi}}{\log d} \left( 1 - \tfrac{1}{2}(\log d)^{-4}-3(\log d)d^{-1/2}\right)\\
            &\geq \frac{\sqrt{2/\pi}}{\log d} \left( 1 - (\log d)^{-4}\right).
    \end{align*}
Let $f$ be the density function of ${\frac{\|\vec x \|_2^2 - \mu_D}{\sigma_D}}$, and let $\chi_I$ be the indicator function of $I$. Since the convolution
    \[ (f \chi_{\pm 1/\log d})\ast \chi_{\pm \delta}\]
is supported on $\pm(1/\log d + \delta)$ and has 1-norm
    \[
    \| f \chi_{\pm 1/\log d})\|_1 \, \| \chi_{\pm \delta}\|_1 \geq 
    \bigg( \frac{\sqrt{2/\pi}}{\log d}\big(1-(\log d)^{-4}\big)\bigg)\,2\delta,\]
there must be some $z$ with
    \begin{align*}
    \big((f \chi_{\pm 1/\log d})\ast \chi_{\pm \delta}\big)(z)
        &\geq \frac{\bigg( \frac{\sqrt{2/\pi}}{\log d}(1-(\log d)^{-4})\bigg)\,2\delta}{2/\log d+2\delta}\\
        &= \delta \sqrt{\frac{2}{\pi}} \left( \frac{1-(\log d)^{-4}}{1+\delta \log d} \right)\\
        &= \bigg(\sqrt{\frac{2}{\pi}}-o(1)\bigg) \, \delta.
    \end{align*}
Consequently, $\vol \annuli(\{1\},1,D,\delta) \geq \left(\sqrt{\frac 2\pi}-o(1)\right)\,2^{-dD}\,\delta$.

Similar calisthenics make the following heuristic argument rigorous. Let $G$ be a normal rv with mean 0 and variance 1:
    \begin{align*}
        \vol(\annuli(A,n,D,\delta)) &\to 2^{-dD} \, \Prob{\vt,\va}{G\in \bigcup_{a\in A} \bigg(-\frac{a-1}{n} \pm \delta\bigg)} \\
            &\geq 2^{-dD}  \, \Prob{\vt,\va}{G \in \big(-1,-1+2\delta |A|\big)} \\
            &= 2^{-dD} \frac{1}{\sqrt{2\pi}} \int_{-1}^{-1+2\delta|A|} \exp(-x^2/2)\,dx \\
            &\geq 2^{-dD} \frac{1}{\sqrt{2\pi}} \exp(-1/2) 2\delta |A| \\
            &> \frac 25 \, 2^{-dD}|A| \delta,
    \end{align*}
where we have used $2\delta\leq 1/n$ to force the intervals $-(a-1)/n \pm \delta$ to be disjoint, and also to force $-1+2\delta |A|<0$. Since the final inequality is strict, we can replace the limit in the central limit theorem with a ``sufficiently large $d$'' hypothesis.
\end{proof}

Lemma~\ref{lem:fewtypes} is not best possible. However, the factor $2^{D+1}$ will turn out to be irrelevant in the final analysis.
\begin{lemma}[There are not many types of progressions]\label{lem:fewtypes}
Assume $k\geq D$. There are fewer than $2^{D+1} N^2$ types of $k$-term progressions with degree at most $D$ contained in $[N]$.
\end{lemma}

\begin{proof}
We suppose that we have a $k$-term progression $a_1,\dots,a_k$ contained in $[N]$ of type $(D',a,b)$, and find restrictions on $D', a$ and $b$. First, fix $D'$. There are clearly at most $N$ possibilities for $a$. It is straightforward to prove by induction that for $\ell\in\{1,\dots,D'\}$
    \[-2^{\ell-1} N < \Delta^\ell (a_i)(v) < 2^{\ell-1} N.\]
Since $\Delta^{D'}(a_i)$ must be a nonzero constant sequence of integers, there are fewer than $2^{D'}N$ possibilities for the constant sequence $(b)=\Delta^{D'}(a_i)$. Summing this total over $1\leq D' \leq D$ yields the claim.
\end{proof}

\section{A base case and an inductive step}

\begin{proposition}[Base Case]\label{prop:base case} If $k>2D$, then as $N\to\infty$
    \begin{equation}\label{equ:allbase}
    \frac{r_{k,D}(N)}{N} \geq \left(\frac{\sqrt {90}}{e \pi^{3/2}}\,\frac{2^D}{D^{1/4}} \binom{2D}{D}-o(1)\right) \,\frac{\sqrt[4]{2\log N}}{2^{\sqrt{8D\log N}}}.
    \end{equation}
\end{proposition}

\begin{proposition}[Inductive Step] \label{prop:inductive step} If $k> 2D$, then there exists a positive constant $C$
    \[\frac{r_{k,D} (N)}{N}\geq C \,2^{-dD} \, \frac{r_{k,2D}(N_0)}{N_0},\]
where
    \[N_0 :=  \frac{e \pi}{3\sqrt 5} \left(4^D \binom{2D}{D}\right)^{-1} \frac{N^{2/d}}{d^{1/2}}.\]
\end{proposition}

Let $A_0$ be a subset of $[N_0]$ with cardinality $r_{k,2D}(N_0)$ that does not contain any $k$-term $2D$-progression, assume $2\delta N_0\leq 2^{-2D}$,
and let
    \[
    A:=A(\vt,\va) = \{n \in [N] \colon n \,\vt + \va \bmod\vec1 \in \annuli(A_0,N_0,D,\delta)\},
    \]
which we will show is typically (with respect to $\vt,\va$ being chosen uniformly from $\BOX_0$) a set with many elements and few types of $D$-progressions. After removing one element from $A$ for each type of progression it contains, we will be left with a set that has large size and no $k$-term $D$-progressions. Since $\BOX_0\times \BOX_0$ has Lebesgue measure 1, this argument could be easily recast in terms of Lebesgue integrals, but we prefer the probabilistic notation and language.

Define $T:=T(\vt,\va)$ to be the set
\begin{equation*}
    \left\{a\in [N] \colon\;
    \begin{matrix}
        \text{$\exists b\in \RR, D'\in[D]$ such that $A(\vt,\va)$ contains} \\
        \text{a $k$-term progression of type $(D',a,b)$}
    \end{matrix}
        \right\},
\label{T.definition2}
\end{equation*}
which is contained in $A(\vt,\va)$. Observe that $A \setminus T$ is a subset of $[N]$ and contains no $k$-term $D$-progressions, and consequently
    \(r_{k,D}(N)\geq |A \setminus T|= |A| - |T|\)
for every $\vt,\va$. In particular,
    \begin{equation}\label{equ:rBound2}
    r_{k,D}(N) \geq \Expect{\vt,\va}{|A| - |T|} = \Expect{\vt,\va}{|A|} - \Expect{\vt,\va}{|T|}.
    \end{equation}

First, we note that
    \begin{equation*}
    \Expect{\vt,\va}{|A|} = \sum_{n=1}^N \Prob{\vt,\va}{n\in A} = \sum_{n=1}^N \Prob{\va}{n\in A}  = N \vol(\annuli(A_0,N_0,D,\delta)).
    \end{equation*}

Let $E{(D',a,b)}$ be 1 if $A$ contains a $k$-term progression of type $(D',a,b)$, and \mbox{$E{(D',a,b)}=0$} otherwise. We have
    \[|T| \leq \sum_{(D',a,b)} E{(D',a,b)},\]
where the sum extends over all types $(D',a,b)$ for which $D'\in[D]$ and there is a $D'$-progression of that type contained in $[N]$; by Lemma~\ref{lem:fewtypes} there are fewer than $2^{D+1} N^2$ such types.

Suppose that $A$ has a $k$-term progression of type $(D',a,b)$, with $D'\in[D]$. Let $p$ be a degree $D'$ polynomial with lead term $p_{D'}\not=0$, $p(1),\dots,p(k)$ a $D'$-progression contained in $A$, and $\Delta^{D'}(p(i))=(b)$. Then
    \[\vec x_i:=p(i)\, \vt +\va \bmod\vec1 \in \annuli(A_0,N_0,D,\delta)\subseteq \BOX_D.\]
By Lemma~\ref{lem:RL2}, the $\vec x_i$ are a $D'$-progression in $\RR^d$, say $\vec P(j)=\sum_{i=0}^{D'} \vec P_i j^i$ has $\vec P(j)=\vec x_j$ and $D'! \vec P_{D'}=D'! p_{D'} \,\vt \bmod\vec 1= b\, \vt \bmod \vec1$. By elementary algebra
    \[Q(j):=\frac{ \|\vec P(j)\|_2^2 - \mu_D}{\sigma_D} - z\]
is a degree $2D'$ polynomial in $j$, and since $\vec P(j)=\vec x_j \in \annuli(A_0,N_0,D,\delta)$ for $j\in [k]$, we know that
    \[Q(j) \in \bigcup_{a\in A_0} \left(-\frac{a-1}{N_0}\pm \delta\right)\]
for all $j\in [k]$, and also $Q(1),\dots,Q(k)$ is a $2D'$-progression. Define the real numbers $a_j \in A_0$, $\epsilon_j \in \pm \delta$ by
    \[Q(j) = - \frac{a_j-1}{N_0} + \epsilon_j.\]

We need to handle two cases separately: either the sequence $(a_i)$ is constant or it is not. Suppose first that it is not constant. Since $a_i\in A_0$, a set without $k$-term $2D$-progressions, we know that $\Delta^{2D+1}(a_i)\not=(0)$, and since $(a_i)$ is a sequence of integers, for some $v$
    \[|\Delta^{2D+1}(a_i)(v)|\geq 1.\]
Consider:
    \[(0)=\Delta^{2D+1}(Q(i)) ={ \frac{1}{N_0} \Delta^{2D+1}(a_i) + \Delta^{2D+1}(\epsilon_i)}, \]
whence
    \[ |\Delta^{2D+1}(\epsilon_i)(v)| = \frac{1}{N_0} |\Delta^{2D+1}(a_i)(v)| \geq \frac{1}{N_0}.\]
Since $|\epsilon_i|<\delta$, we find that $|\Delta^{2D+1}(\epsilon_i)(v)| < 2^{2D+1}\delta$, and since we assumed that $2\delta N_0\leq 2^{-2D}$, we arrive at the impossibility
    \[ \frac{1}{N_0} \leq |\Delta^{2D+1}(\epsilon_i)(v)| < 2^{2D+1} \delta \leq  2^{2D}\,\cdot\,\frac{2^{-2D}}{N_0} = \frac{1}{N_0}.\]

Now assume that $(a_i)$ is a constant sequence, say $a:=a_i$, so that
    \[Q(j) \in -\frac{a-1}{N_0} \pm \delta \]
for all $j\in[k]$. This translates to
    \[\|\vec P(j)\|_2^2 \in  \mu_D-(z-\frac{a-1}{N_0})\sigma_D \pm \delta\sigma_D.\]
Using Lemma~\ref{lem:RL1}, the lead coefficient $\vec P_{D'}$ of $\vec P(j)$ satisfies
    \begin{multline*}
    \|D'! \vec P_{D'}\|_2 \leq D'!\,2^{D'} {(2D')!}^{-1/2} \sqrt{\delta \sigma_D} \leq
    D!\,2^{D} {(2D)!}^{-1/2} \sqrt{\delta \sigma_D} \\
    = \left(\frac{4^D \, D!^2}{(2D)!}\right)^{1/2} \sqrt{ \sigma_D \delta }
        = \sqrt{ F \sigma_D \delta },
    \end{multline*}
where $F:=4^D/\binom{2D}{D}$. We have deduced that $E{(D',a,b)}=1$ only if
    \[
    a\,\vt+\va \bmod 1 \in \annuli(A_0,N_0,D,\delta) \quad \text{and} \quad \|b\, \vt  \bmod 1\|_2 \leq \sqrt{ F \sigma_D \delta }.
    \]
Since $\va$ is chosen uniformly from $\BOX_0$, we notice that
    \[
    \Prob{\va}{a\,\vt+\va\bmod 1 \in \annuli(A_0,N_0,D,\delta)}=\vol \annuli(A_0,N_0,D,\delta),
    \]
independent of $\vt$. Also, we notice that the event $\{\|b\, \vt \bmod 1\|_2 \leq \sqrt{ F \sigma_D \delta }\}$ is independent of $\va$, and that since $b$ is an integer, $\vt\bmod\vec 1$ and $b\,\vt\bmod \vec 1$ are identically distributed. Therefore, the event $\{\|b\, \vt \bmod 1\|_2 \leq \sqrt{ F \sigma_D \delta }\}$ has probability at most\footnote{In fact, since we will shortly choose $\delta$ so that $F\sigma_D\delta \to 0$, this upper bound cannot be improved.}
    \[\vol \ball(\sqrt{ F \sigma_D \delta }) = \frac{2\pi^{d/2} (\sqrt{ F \sigma_D \delta })^{d}}{\Gamma(d/2) d},\]
where $\ball(x)$ is the $d$-dimensional ball in $\RR^d$ with radius $x$. It follows that
    \[\Prob{\vt,\va}{E{(D',a,b)}=1}\leq \vol \annuli(A_0,N_0,D,\delta) \cdot\vol \ball(\sqrt{ F \sigma_D \delta }),\]
and so
    \[\Expect{\vt,\va}{|T|} \leq 2^{D+1} N^2 \vol \annuli(A_0,N_0,D,\delta)  \cdot \vol \ball(\sqrt{ F \sigma_D \delta }).\]

Equation~\eqref{equ:rBound2} now gives us
    \[\frac{r_{k,D}(N)}{N} \geq \vol(\annuli(A_0,N_0,D,\delta)) \left( 1- 2^{D+1} N  \vol \ball(\sqrt{ F \sigma_D \delta })\right).\]
Setting
    \[
    \delta := \frac{1}{\pi F} \,\left(\frac{d}{(d+2)2^{D+1}}\right)^{2/d}\, \frac{\Gamma(d/2)^{2/d}}{N^{2/d} \sigma_D}\sim \frac{3 \sqrt{5}}{e \pi} \binom{2D}{D} \frac{d^{1/2}}{N^{2/d}},
    \]
we observe that
    \[
    1- 2^{D+1} N \frac{2\pi^{d/2} (F \delta^{1/2} d^{1/4})^{d}}{\Gamma(d/2) d}= \frac {d}{d+2}.
    \]

\subsection{Finish proof of Proposition~\ref{prop:base case}}
We set
         \[d:=\left\lfloor\sqrt{\frac{2 \log N}D}\right\rfloor,\]
so that $\delta\log d\to0$, and
    \begin{align*}
    \frac{r_{k,D}(N)}{N}
        & \geq \frac{d}{d+2}\,  \vol \annuli(\{1\},1,D,\delta)\\
        &\geq \frac{d}{d+2}  \,\left(\sqrt{\frac 2\pi}-o(1)\right) \,  2^{- dD} \,\delta \\
        &\geq \frac{d}{d+2} \,\left(\sqrt{\frac 2\pi}-o(1)\right) \,  2^{- dD} \,\frac{1}{\pi F} \,\left(\frac{d}{(d+2)2^{D+1}}\right)^{2/d}\, \frac{\Gamma(d/2)^{2/d}}{N^{2/d} \sigma_D} \\
        &= \left(\frac{\sqrt 2}{\pi^{3/2} F}-o(1)\right)\,  2^{- dD} \, \frac{\Gamma(d/2)^{2/d}}{N^{2/d} \sigma_D} \\
        &= \left(\frac{\sqrt 2}{\pi^{3/2} F}-o(1)\right)\,  2^{- dD} \, \frac{(1+o(1))d/2e}{N^{2/d} 2^{-2D}\sqrt{d/180}}\\
        &\geq \left(\frac{2^{2D}\,\sqrt {360}}{2e \pi^{3/2} F}-o(1)\right) \,  2^{- dD}\, \frac{\sqrt d}{N^{2/d}}\\
        &= \left(\frac{\sqrt {90}}{e \pi^{3/2}}\, \binom{2D}{D}-o(1)\right) \,\sqrt{d}\, \exp\big(-(dD+\frac 2d \log N)\big).
    \end{align*}
Define the error term $\epsilon(N)$ by
    \[
    dD+\frac2d \log N
        = \sqrt{8D \log N}+\epsilon(N),
    \]
and observe that for any integer $\ell$, we have $\epsilon(x)$ monotone increasing on $[2^{\ell^2 D/2},2^{(\ell+1)^2D/2})$, while $N$ being in that interval gives $d=\ell$. By algebra $\epsilon(2^{\ell^2 D/2})=0$, and also
    \[\lim_{N \to \exp({(d+1)^2D/2})} \epsilon(N) = \frac Dd.\]
It follows that $\epsilon(N) \leq D/(\sqrt{2(\log N)/D}-1)$.

From this, we see that
    \begin{align*}
    \exp\big(-(dD+\frac 2d \log N)\big)
        &\geq \exp(-\sqrt{8D\log N}) \exp\left(\frac{D}{\sqrt{2(\log N)/D}-1}\right)\\
        &=(1+o(1))\exp(-\sqrt{8D\log N}),
    \end{align*}
which completes the proof of Proposition~\ref{prop:base case}.

\subsection{Finish proof of Proposition~\ref{prop:inductive step}}
We set
    \[
    N_0 := \frac{e \pi}{3\sqrt 5} \left(4^D \binom{2D}{D}\right)^{-1} \frac{N^{2/d}}{d^{1/2}}
    \]
which accomplishes $\frac 14 2^{-2D} \leq 2\delta N_0 \leq 2^{-2D}$.
With this $\delta, N_0$ and Lemma~\ref{lem:annuli} we have,
    \begin{align*}
    \frac{r_{k,D}(N)}{N}
        & \geq \frac {d}{d+2}  \vol \annuli(A_0,N_0,D,\delta)\\
        &\geq \frac {d}{d+2}  \,\frac 25 \,2^{-dD}\,|A_0|\,\delta \\
        &\geq \frac12\,\frac25\,2^{-dD} \frac{|A_0|}{N_0}\,\delta N_0 \\
        &= C \,2^{-dD} \, \frac{r_{k,2D}(N_0)}{N_0}.
    \end{align*}

\section{Proof of Theorem~\ref{thm:main}}

We proceed by induction, with the base case of $n=2$ following immediately from Proposition~\ref{prop:base case}. We now assume that Theorem~\ref{thm:main} holds for $n$, assume that $k>2^nD$, and show that
    \[\frac{r_{k,D}(N)}{N} \geq C \frac{(\log N)^{1/(2n+2)}}{\exp\big((n+1)2^{n/2}D^{n/(n+1)} \sqrt[n+1]{\log N}\big)}.\]

By Proposition~\ref{prop:inductive step}, we have
    \begin{equation*}
    \frac{r_{k,D}(N)}{N}
        \geq C \frac{1}{2^{dD}}\, \frac{r_{k,2D}(N_0)}{N_0},
    \end{equation*}
with $N_0=C N^{2/d}d^{-1/2}$. Since $k>2^n D = 2^{n-1} (2D)$, the inductive hypothesis gives us
    \begin{align*}
    \frac{r_{k,D}(N)}{N}
        &\geq C \frac{1}{2^{dD}} \frac{(\log N_0)^{1/(2n)}}{\exp\big(n 2^{(n-1)/2} (2D)^{(n-1)/n} \sqrt[n]{\log N_0}\big)}\\
        &= C \frac{(\log N_0)^{1/(2n)}}{\exp\big(dD+n 2^{(n-1)/2} (2D)^{(n-1)/n} \sqrt[n]{\log C-\frac12\log d+\frac 2d \log N}\big)}\\
        &\geq C \frac{(\log N_0)^{1/(2n)}}{\exp\big(dD+n 2^{(n-1)/2} (2D)^{(n-1)/n} \sqrt[n]{\frac 2d \log N}\big)},
    \end{align*}
with the final inequality coming from $d$ being sufficiently large.

Setting
    \[
    d:= \Floor{2^{n/2}   \left( \frac{\log N}{D}\right)^{1/(n+1)}}
    \]
we arrive at the error term and bound for it:
    \begin{equation*}
    dD+n2^{(n-1)/2} (2D)^{(n-1)/n} \sqrt[n]{\frac 2d \log N}
         = (n+1) 2^{n/2} D^{n/(n+1)} (\log N)^{1/(n+1)}  +\epsilon(N)
    \end{equation*}
where
    \[\epsilon(N)\leq (1+o(1)) \frac{(n+1)D^{(n+2)/(n+1)}}{n\,2^{n/2+1} \,(\log N)^{1/(n+1)}} \leq \frac{C}{(\log N)^{1/(n+1)}}.\]
Thus,
    \begin{multline*}
    \exp\!\bigg(-\big(dD+n2^{(n-1)/2} (2D)^{(n-1)/n} \sqrt[n]{\frac 2d \log N} \big) \bigg) \geq \\
        (1+o(1))\exp\!\left( -(n+1) 2^{n/2} D^{n/(n+1)} (\log N)^{1/(n+1)} \right).
    \end{multline*}

\section{Further Thoughts}
The approach here works {\it mutatis mutandis} for constructing a subset of an arbitrary set ${\cal N}$ of $N$ integers. The number of progressions in ${\cal N}$ becomes a critical parameter, and the inductive step is somewhat more technical. The specific changes are detailed in~\cite{ProgressionsInSubsets}.

Further, the methods here can serve as a basic outline for constructing thick subsets of a large arbitrary set that does not contain nontrivial solutions to a linear system of equations. This problem has seen recent progress due to Shapira~\cite{Shapira}, but a universal thick construction remains elusive.

\end{document}